\documentclass[11pt]{amsart}

\usepackage{amsmath,amssymb,amsthm,amsrefs,a4wide}
\usepackage{latexsym}
\usepackage{indentfirst}
\usepackage{graphicx}
\usepackage{placeins}
\usepackage{booktabs}
\usepackage{algorithm}
\usepackage{algorithmic}
\usepackage{multirow}
\usepackage{color}

\theoremstyle{plain}
\newtheorem{theorem}{Theorem}

\theoremstyle{definition}
\newtheorem{defn}[theorem]{Definition}
\newtheorem{claim}[theorem]{Claim}
\newtheorem{assum}[theorem]{Assumption}

\theoremstyle{remark}

\newcommand{\bbm}{\begin{bmatrix}}
\newcommand{\ebm}{\end{bmatrix}}

\begin{document}

\title[Multimodal Sampling via Approximate Symmetries]{Multimodal Sampling via Approximate Symmetries}

\author[]{Lexing Ying} \address[Lexing Ying]{Department of Mathematics, Stanford University,
  Stanford, CA 94305} \email{lexing@stanford.edu}

\thanks{The author thanks Professors Persi Diaconis and Wing Hung Wong for suggestions and comments.}

\keywords{Multimodal distributions, symmetries, annealed importance sampling, tempered transition, Ising models.}

\subjclass[2010]{82B20,82B80.}

\begin{abstract}
  Sampling from multimodal distributions is a challenging task in scientific computing. When a distribution has an exact symmetry between the modes, direct jumps among them can accelerate the samplings significantly. However, the distributions from most applications do not have exact symmetries. This paper considers the distributions with approximate symmetries. We first construct an exactly symmetric reference distribution from the target one by averaging over the group orbit associated with the approximate symmetry. Next, we can apply the multilevel Monte Carlo methods by constructing a continuation path between the reference and target distributions. We discuss how to implement these steps with annealed importance sampling and tempered transitions. Compared with traditional multilevel methods, the proposed approach can be more effective since the reference and target distributions are much closer. Numerical results of the Ising models are presented to illustrate the efficiency of the proposed method.
\end{abstract}

\maketitle

\section{Introduction}\label{sec:intro}


Sampling from multimodal distributions is a challenging task in scientific computing. Standard Markov Chain Monte Carlo (MCMC) methods face difficulty since they tend to mix efficiently locally but fail to introduce moves between different modes. Over the years, several classes of methods have been introduced to address this issue.

The first class is the multilevel or tempering methods, where the target distribution is associated with a fixed temperature. The high-temperature version of the target distribution is typically easy-to-sample. This class constructs a continuation path between the high-temperature version and the target distribution. The actual sampling is conducted via Metropolis-Hasting moves along the continuation path. Well-known examples of this type include simulated tempering, parallel tempering, tempered transitions, and annealed importance sampling. The efficiency of these methods depends on the acceptance rate of the MCMC moves or the variance of the importance weights. To keep these under control, the continuation path sometimes needs to be long and discretized with many intermediate temperatures. This could result in high computational and storage costs. In the rest of this paper, we will refer to these methods as continuation methods because the concept of connecting an easy-to-sample distribution with the target distribution is more general than tempering.

A second class of methods uses the symmetry of the target distribution. For many spin systems, an explicit symmetry exists between the spin values. One can accelerate the mixing by introducing a spin-value flip over a large cluster of spins. This is, for example, used partly in the Swendsen-Wang algorithm, where the global move flips between $+1$ and $-1$ configurations. Another example is the Hamiltonian Monte Carlo (or hybrid Monte Carlo), where the Hamiltonian flow serves as a symmetry that preserves the distribution in the phase space.

Though distributions with exact symmetry are important for analytical studies and theoretical understanding, many distributions from practical applications often fail to have exact symmetries. A natural question is whether it is still possible to leverage the power of symmetries there.

A key observation is that many high-dimensional multimodal distributions in computational physics can often be {\em approximately symmetric}. This is quite common at relatively low temperatures because one of the modes will otherwise dominate, effectively reducing the distribution to a single mode. For a distribution to be genuinely multimodal at low temperatures, there is often an approximate symmetry behind the scene.

\subsection{Contribution} This paper introduces a method for sampling from multimodal distributions with approximate symmetries. First, we identify the approximate symmetry. Second, we construct a symmetric reference multimodal distribution by averaging the log density over the group orbit associated with the approximate symmetry. Third, by applying the continuation method between the symmetric reference distribution and the target one, we obtain efficient algorithms for the target multimodal distribution. Figure \ref{fig:idea} illustrates these main steps. 
\begin{figure}[h!]
  \centering \includegraphics[scale=0.35]{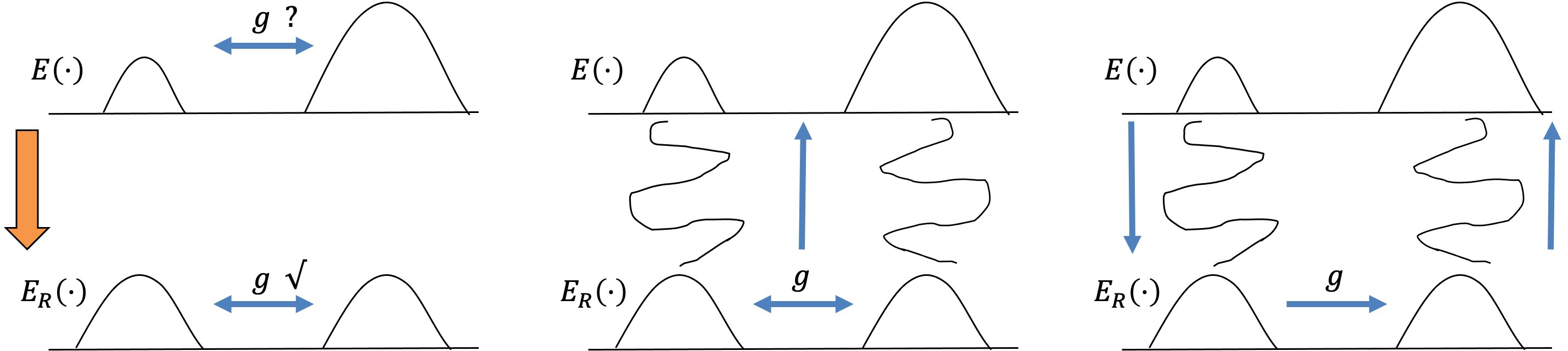}
  \caption{The main steps of the algorithms. Left: $G$ is the group of symmetries. $G$ is an approximate symmetry for the distribution with log density $E(s)$. Through averaging over the group orbit, we construct a reference log density $E_R(s)$, which is symmetric under $G$. Middle: annealed importance sampling (AIS), where $g\in G$ allows efficient sampling from the reference $E_R(s)$. Right: tempered transition (TT), where $g\in G$ allows for global jumps within a TT move.}
  \label{fig:idea}
\end{figure}

Compared with the existing continuation method that uses a high-temperature distribution as the reference, the proposed method can be more efficient. The reason is that the distance between the target distribution and our reference distribution constructed from group action averaging is often much shorter. This results in a few intermediate steps along the continuation path and less computational cost. Figure \ref{fig:cmp} offers a cartoon illustration of this point.

\begin{figure}[h!]
  \centering
  \includegraphics[scale=0.3]{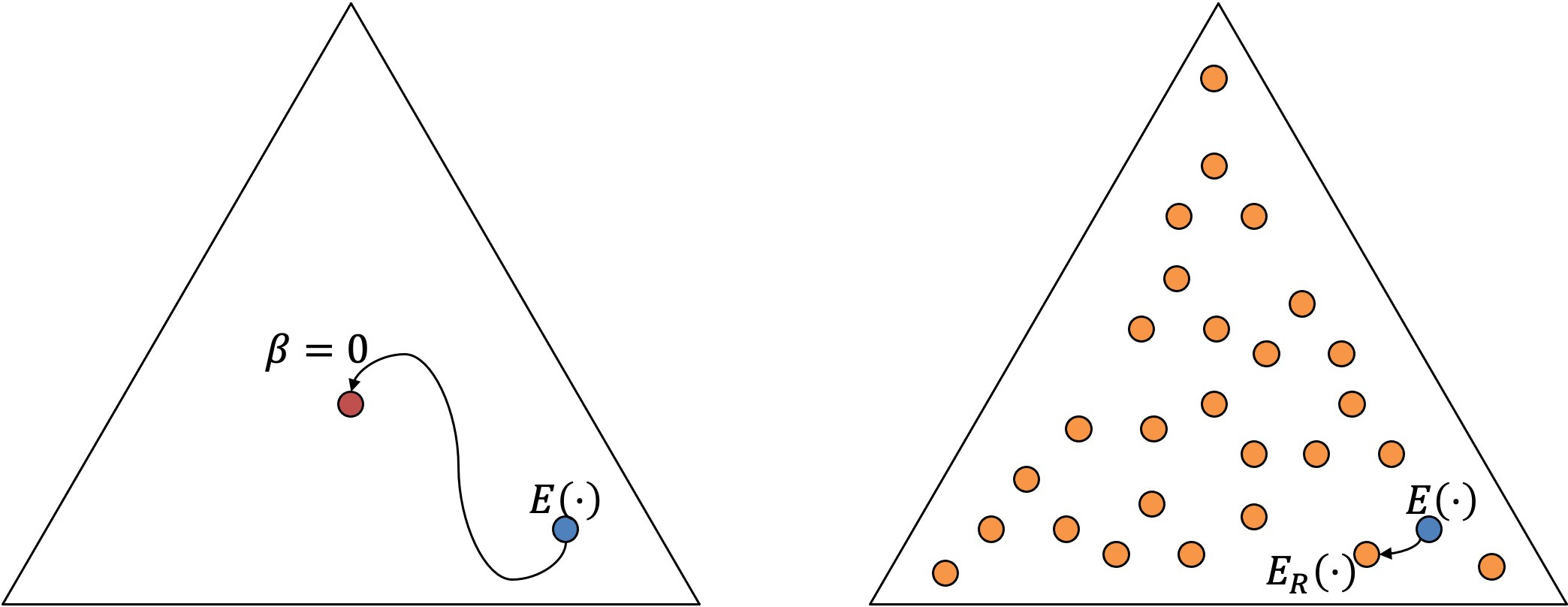} 
  \caption{The triangle stands for the probability simplex defined over the state space. Left: the usual continuation method. The usual reference distribution with $\beta=0$ has a lot of symmetries and is easy to sample. However, it is often far from the target distribution $E(s)$, so the continuation path is long. Right: the proposed method. Based on the approximate symmetry $G$ of the $E(s)$, we identify a nearby reference distribution $E_R(s)$ with $G$ as the exact symmetry. The orange dots are the distributions that satisfy the exact symmetry in the group $G$ ($\beta=0$ is merely one of them). The continuation path from $E(s)$ to $E_R(s)$ is much shorter than the one from $E(s)$ to the $\beta=0$ distribution.}
  \label{fig:cmp}
\end{figure}

\subsection{Related work}
The field of Monte Carlo methods is vast. The standard references for advanced Monte Carlo strategies are \cite{liu2001monte,liang2011advanced}. In what follows, we only touch on the references closely related to our work.

In the class of the continuation (or tempering/annealing) methods, several well-known examples include simulated tempering \cite{marinari1992simulated,geyer1995annealing}, parallel tempering \cite{geyer1991markov,hukushima1996exchange}, simulated tempering \cite{neal1996sampling}, and annealed importance sampling \cite{neal2001annealed}.

Many methods take advantage of the exact symmetries and group actions. For example, within the clustering approach of spin systems, the Swendsen-Wang and Wolff algorithms \cite{swendsen1987nonuniversal,wolff1989collective,edwards1988generalization} utilizes the spin-value symmetry in their cluster moves across all spatial scales. As we mentioned earlier, Hamiltonian or hybrid Monte Carlo \cite{duane1987hybrid,neal1996monte} is another example that utilizes the symmetry of the Hamiltonian flow. The Hamiltonian Monte Carlo approach has also been extended to discrete state spaces in \cite{diaconis2000analysis,kapfer2017irreversible}. Group action has also been applied in \cite{liu2000generalised} to improve Gibbs sampling.

Mode-jumping (or mode-hopping) is another class of methods for multimodal sampling \cite{tjelmeland2001mode,andricioaei2001smart,sminchisescu2003mode,tjelmeland2004use,sminchisescu2007generalized,ibrahim2009new}. Assuming that the modes are pre-identified, such a method constructs an easy-to-sample local approximation around each mode and uses their sum (typically a mixture of Gaussians) to approximate the overall target distribution. This approximation is then utilized in the Metropolis-Hasting step for sampling. Compared to the methods based on symmetries and approximate symmetries, this class of methods typically requires knowing more information about the modes. The numerical performance depends critically on the accuracy of the local approximation, which can be a non-trivial task.

Another important class is the population-based methods. Two well-known examples of this class are the Snooker algorithm \cite{gilks1994adaptive} and the evolutionary Monte Carlo \cite{liang2001evolutionary}. A recent development emphasizing affine invariance can be found in \cite{goodman2010ensemble}.

\subsection{Content}

The result of the paper is organized as follows. Section \ref{sec:method} describes the general framework and shows how to combine it with two continuation methods (annealed importance sampling and tempered transition). Section \ref{sec:ex} focuses on two examples from the Ising model. Finally, Section \ref{sec:disc} contains some discussions for future work.

\section{Method}\label{sec:method}

\subsection{Approximate symmetry}
We first describe how exact symmetry can help accelerate the mixing and then how to extend it to multimodal distributions with approximate symmetries.

Let $S$ be state space, $p(s)$ be the target probability at $s$, and $E(s)$ be its log density (up to a normalization constant), i.e., $p(s) \propto \exp(E(s))$.  We define symmetry via a group $G$ that acts on the state space on the left. Following the standard notations in group theory (see \cite{vinberg2003course} for example), we use $gs$ to stand for the state obtained by acting $g$ on $s$. 

\begin{defn}
  The distribution $p(x)\propto \exp(E(s))$ is symmetric with respect to $G$ if
  \begin{equation}
    E(gs) = E(s) \label{eq:Egs=Es}
  \end{equation}
  for any $g\in G$ and $s\in S$.
\end{defn}

This paper assumes that the symmetry can map from one mode to another. For simplicity, let us assume that $G$ is a discrete group, though the Lie group case is similar. In the simplest case, $G=\{e, g\}$ with $g^2 = e$. Given $G$, one can define a Markov chain with the transition kernel $T(s,t)$ where
\begin{equation}
  T(s,t) =
  \begin{cases}
    \frac{1}{|G|-1} & t=gs \;\text{for some}\; g\not=e,\\
    0               & \text{otherwise}.
  \end{cases}\label{eq:Tgroup}
\end{equation}
The claim is that $T$ satisfies the detailed balance \cite{liu2001monte}, i.e.,
\[
p(s) T(s,t) = p(t) T(t,s).
\]
To see this, notice that when $t$ cannot be written as $gs$ for any $g$, both sides are zero. When $t=s$, both sides are zero too. Finally, when $t=gs$ for some $g\not=e$, $p(s) \cdot \frac{1}{|G|-1} = p(t) \cdot \frac{1}{|G|-1}$ following \eqref{eq:Egs=Es}.

The more interesting case is when the symmetry in $G$ is only approximate. A distribution $p(x)\propto \exp(E(s))$ is {\em approximately symmetric} with respect to $G$ if
\begin{equation}
  E(gs) \approx E(s) \label{eq:EgsappEs}
\end{equation}
for any $g\in G$ and $s\in S$. The following assumption is important for the proposed approach.

\begin{assum}\label{as:transitive}
  The action of $G$ is transitive over the modes of $p(x)\propto \exp(E(x))$.
\end{assum}

\begin{defn}\label{def:ref}
  The {\em reference} log density $E_R(s)$ is defined by averaging over the group orbit
  \begin{equation}
    E_R(s) := \frac{1}{|G|} \sum_{g\in G} E(gs). \label{eq:Ers}
  \end{equation}
  The corresponding reference distribution is $p_R(s) \propto \exp(E_R(s))$.
\end{defn}

Three properties of $E_R(s)$ are important for the rest of the paper.
\begin{itemize}
\item $E_R(s)$ is symmetric under $G$, i.e.,
  \[
  E_R(s) = E_R(gs).
  \]
  This is because $E_R(s) = \frac{1}{|G|} \sum_h E(hs) = \frac{1}{|G|} \sum_h E(h\cdot gs) = E_R(gs)$, due to the orbit averaging.
\item $E_R(s)$ and $E(s)$ are close:
  \[
  E_R(s) \approx E(s)
  \]
  and this follows from averaging \eqref{eq:EgsappEs}.
\item $G$ also acts transitively on the modes of $E_R(s)$.
\end{itemize}

With these three properties, the reference log density $E_R(s)$ can help the sampling from the multimodal distribution $p(s)\propto \exp E(s)$. The following two subsections describe how this can done in two continuation methods (annealed importance sampling and tempered distribution).

\subsection{Annealed importance sampling}

\subsubsection{Existing method}

In annealed importance sampling \cite{neal2001annealed} (AIS), one has $E(s) = -\beta U(s)$ where $\beta$ is the target inverse temperature and $U(s)$ is the potential. The method introduces the inverse-temperature sequence:
\[
0 \approx \beta_0 < \beta_1 < \ldots < \beta_{L-1} < \beta_L = \beta.
\]
Each level $l$ is associated with a distribution $p_l(s) \propto \exp(-\beta_l U(s))$ at inverse temperature $\beta_l$.

Since $\beta_0\approx 0$, one assumes that $p_0(s) \propto \exp(-\beta_0 U(s))$ is easy-to-sample. We further assume that, at each $\beta_l$, $p_l(s)$ associated with a detailed balanced Markov Chain with a transition kernel $T_l(s,t)$, i.e.,
\begin{equation}
  p_l(s) T_l(s,t) = p_l(t) T_l(t,s).  \label{eq:db}
\end{equation}

Given these building blocks, AIS computes independent weighted samples, one by one, from the following procedure.
\begin{enumerate}
\item Sample a configuration $s_{1/2}$ from $p_0(s)\propto \exp(-\beta_0 U(s))$.
\item For $l=1,\ldots,L-1$, take one step (or a few steps) of $T_l(\cdot,\cdot)$ (associated with the distributions $\beta_l$) from $s_{l-1/2}$ to $s_{l+1/2}$.
\item Set $s := s_{L-1/2}$.
\item Compute the weight
  \[
  w := \frac{p_1(s_{1/2})}{p_0(s_{1/2})} \cdots \frac{p_L(s_{1/2})}{p_{L-1}(s_{1/2})} \propto
  \frac{\exp(-\beta_1 U(s_{1/2}))}{\exp(-\beta_0 U(s_{1/2}))}  \cdot  \frac{\exp(-\beta_L U(s_{L-1/2}))}{\exp(-\beta_{L-1} U(s_{L-1/2}))}.
  \]
\end{enumerate}
The normalization constant hidden in $\propto$ is the same for all independent samples from AIS and hence can be ignored.

\begin{claim}\label{cl:ais}
  The configuration $s$ with weight $w$ samples the target distribution $p(s) \propto \exp E_L(s) =\exp E(s)$.
\end{claim}

\begin{proof}
  To see this, consider the path $(s_{1/2}, \ldots, s_{L-1/2})$. This path is generated with probability
\[
p_0(s_{1/2}) T_1(s_{1/2},s_{3/2}) \cdots T_{L-1}(s_{L-3/2},s_{L-1/2}).
\]
Multiplying this with $w$ and using the detailed balance \eqref{eq:db} of $T_l$ gives
\begin{align*}
  & \, p_0(s_{1/2}) T_1(s_{1/2},s_{3/2}) \cdots T_{L-1}(s_{L-3/2},s_{L-1/2}) \cdot \frac{p_1(s_{1/2})}{p_0(s_{1/2})} \cdots \frac{p_L(s_{L-1/2})}{p_{L-1}(s_{L-1/2})}\\
  = & \, p_L(s_{L-1/2}) T_{L-1}(s_{L-1/2},s_{L-3/2}) \cdots T_1(s_{3/2},s_{1/2}),
\end{align*}
which is the probability of going backward, i.e., starting from a sample $s_{L-1/2}$ of $p_L(\cdot) \equiv p(\cdot)$. Taking the margin of the last slot $s_{L-1/2}$ proves that $s:=s_{L-1/2}$ with weight $w$ samples the distribution $p_L(s)\equiv p(s)$.
\end{proof}

\subsubsection{Extension using symmetric reference}\label{sec:pmais}

The proposed method does not change the $\beta$ value. Instead, we continue between $E(s)$ and the symmetric reference $E_R(s)$.

Let $c(t):[0,1]\rightarrow[0,1]$ be a strictly monotone function with $c(0)=0$ and $c(1)=1$. Define $E_l(s)$ for $0\le l \le L$:
\[
E_l(s) = (1-c(l/L)) E_R(s) + c(l/L) E(s).
\]
In the simplest case, $c(t)=t$. This sequence goes from the reference distribution $E_0(s)=E_R(s)$ to the target one $E_L(s)=E(s)$. Since $E_R(s)\approx E(s)$, the discretization length $L$ can be small.

Since $E_R(s)$ is symmetric and $G$ acts transitively on its modes, it can be sampled efficiently. At each $E_l(s)$, we assume that $p_l(s) \propto \exp(E_l(s))$ is associated with a detailed balanced Markov Chain $T_l(s,t)$.

Given these building blocks, the proposed method computes independent weighted samples from the following procedure.
\begin{enumerate}
\item Sample a configuration $s_{1/2}$ from $E_0(\cdot)=E_R(\cdot)$. This can be done efficiently since $E_R(\cdot)$ is symmetric.
\item For $l=1,\ldots,L-1$, take one step (or a few steps) of $T_l(\cdot,\cdot)$ (associated with the log density $E_l(\cdot)$) from $s_{l-1/2}$ to $s_{l+1/2}$.
\item Set $s := s_{L-1/2}$.
\item Compute the weight
  \[
  w := \frac{p_1(s_{1/2})}{p_0(s_{1/2})} \cdots \frac{p_L(s_{1/2})}{p_{L-1}(s_{1/2})} \propto
  \frac{\exp(E_1(s_{1/2}))}{\exp(E_0(s_{1/2}))} \cdots \frac{\exp(E_L(s_{L-1/2}))}{\exp(E_{L-1}(s_{L-1/2}))}.
  \]
\end{enumerate}
The justification for its correctness is the same as Claim \ref{cl:ais}.

\subsection{Tempered transition}

\subsubsection{Existing method}

In tempered transition \cite{neal1996sampling} (TT), one introduces the following sequence of inverse temperatures:
\[
\beta = \beta_0 > \ldots > \beta_L \approx 0 < \ldots < \beta_{2L} = \beta,
\]
with $\beta_l = \beta_{2L-l}$ for $0\le l \le L$.

Since $\beta_L\approx 0$, one assumes that $p_L(s) \propto \exp(-\beta_L U(s))$ can be mixed rapidly. At each $\beta_l$, $p_l(s) \propto \exp(-\beta_l U(s))$ is again associated with a detailed balanced Markov Chain with a kernel $T_l(s,t)$. In addition, $T_l(s,t) = T_{2L-l}(s,t)$.

The tempered transition (TT) combines the usual local moves at $\beta$ with the following more sophisticated one.
\begin{enumerate}
\item Start from a configuration $s$ at $\beta$. Set $s_{1/2}=s$.
\item For $l=1,\ldots,L-1$, take one step (or a few steps) of $T_l(\cdot,\cdot)$ (associated with the distributions at $\beta_l$) from $s_{l-1/2}$ to $s_{l+1/2}$.
\item At $l=L$, take one step (or a few steps) of $T_L(\cdot,\cdot)$ (associated with the distributions $\beta_L$) from $s_{L-1/2}$ to $s_{L+1/2}$. Since $\beta_L \approx 0$, this step can move efficiently between different modes.
\item For $l=L+1,\ldots,2L-1$, take one step (or a few steps) of $T_l(\cdot,\cdot)$ (associated with the distributions $\beta_l$) from $s_{l-1/2}$ to $s_{l+1/2}$.
\item Set $t := s_{2L-1/2}$.
\item Compute the amplitude
  \begin{align*}
    A  & = \min \left(1, \frac{p_1(s_{1/2})}{p_0(s_{1/2})} \cdots \frac{p_{2L}(s_{2L-1/2})}{p_{2L-1}(s_{2L-1/2})} \right)\\
    & = \min \left(1,\frac{\exp(-\beta_1 U(s_{1/2}))}{\exp(-\beta_0 U(s_{1/2}))}  \cdots  \frac{\exp(-\beta_{2L} U(s_{2L-1/2}))}{\exp(-\beta_{2L-1} U(s_{2L-1/2}))} \right),
  \end{align*}
  where the normalization constants cancel out.
\item Sample $\alpha$ uniformly from $[0,1]$. If $\alpha<A$, set $s=t$. Otherwise, keep $s$ unchanged. 
\end{enumerate}

\begin{claim}\label{cl:tt}
  The tempered transition satisfies the detailed balance.
\end{claim}

\begin{proof}
Let us consider the total transition between $s$ and $t$. Each transition is implemented with a path $s_{1/2}, \ldots, s_{2L-1/2}$. Consider such a path one by one.The probability flux from $s$ to $t$ is given by
\[
p_0(s_{1/2}) T_1(s_{1/2},s_{3/2}) \cdots T_{2L-1}(s_{2L-3/2},s_{2L-1/2}),
\]
while the probability flux from $t$ to $s$ is given by
\[
\begin{aligned}
  & p_{2L}(s_{2L-1/2}) T_{1}(s_{2L-1/2},s_{2L-3/2}) \ldots T_{2L-1}(s_{1/2},s_{3/2})\\
  =&p_{2L}(s_{2L-1/2}) T_{2L-1}(s_{2L-1/2},s_{2L-3/2}) \ldots T_{1}(s_{1/2},s_{3/2}),
\end{aligned}
\]
due to the symmetry $T_l =T_{2L-l}$. The ratio of the latter over the former is
\[
\frac{p_{2L}(s_{2L-1/2}) T_{2L-1}(s_{2L-1/2},s_{2L-3/2}) \ldots T_{1}(s_{1/2},s_{3/2})}{p_0(s_{1/2}) T_1(s_{1/2},s_{3/2}) \cdots T_{2L-1}(s_{2L-3/2},s_{2L-1/2})} =
\frac{p_1(s_{1/2})}{p_0(s_{1/2})} \cdots  \frac{p_{2L}(s_{2L-1/2})}{p_{2L-1}(s_{2L-1/2})}.
\]
Therefore, the Metropolis-Hasting move described above satisfies the detailed balance.
\end{proof}

\subsubsection{Extension using symmetric reference}\label{sec:pmtt}

Instead of changing the value of $\beta$, we continue between $E(\cdot)$ and its symmetric reference $E_R(\cdot)$.

Let $c(t):[0,1]\rightarrow[0,1]$ again be a strictly monotone function with $c(0)=0$ and $c(1)=1$. Define $E_l(s)$ for $0\le l \le 2L$. For $0\le l \le L$
\[
E_l(s) = (1-c(l/L)) E(s) + c(l/L) E_R(s),
\]
and for $L \le l \le 2L$
\[
E_l(s) = E_{2L-l}(s).
\]
In the simplest case, $c(t)=t$. This sequence starts from the target distributions $E_0(s)=E(s)$, goes to the reference one $E_L(s)=E_R(s)$, and returns to $E_{2L}(s)=E(s)$. Since $E_R(s)\approx E(s)$, the discretization length $2L$ is small.

Since $E_R(s)$ is symmetric under $G$, it enjoys the group transition \eqref{eq:Tgroup} between multiple modes. At each $E_l(s)$, we assume that $p_l(s) \propto \exp(E_l(s))$ is associated with a detailed balanced Markov chain $T_l(s,t)$.

The proposed method combines the usual local moves at $E(\cdot)$ with the following more sophisticated move that takes advantage of the symmetric reference $E_R(s)$.
\begin{enumerate}
\item Start from the current configuration $s$ at $E(\cdot)$. Set $s_{1/2}=s$.
\item For $l=1,\ldots,L-1$, take one step (or a few steps) of $T_l(\cdot,\cdot)$ (associated with
  the distribution at $E_l$) from $s_{l-1/2}$ to $s_{l+1/2}$.
\item At $l=L$, take a random $g\not=e$ uniformly from $G$. Let $s_{L+1/2} = g s_{L-1/2}$ be the resulting configuration (see \eqref{eq:Tgroup}).
\item For $l=L+1,\ldots,2L-1$, take one step (or a few steps) of $T_l(\cdot,\cdot)$ (associated with the distribution at $E_l$) from $s_{l-1/2}$ to $s_{l+1/2}$.
\item Set $t := s_{2L-1/2}$.
\item Compute the amplitude
  \begin{align*}
    A  & = \min \left(1, \frac{p_1(s_{1/2})}{p_0(s_{1/2})} \cdots \frac{p_{2L}(s_{2L-1/2})}{p_{2L-1}(s_{2L-1/2})} \right)\\
    & = \min \left(1,\frac{\exp(E_1(s_{1/2}))}{\exp(E_0(s_{1/2}))}  \cdots  \frac{\exp(E_{2L}(s_{2L-1/2}))}{\exp(E_{2L-1}(s_{2L-1/2}))} \right).
  \end{align*}
\item Sample $\alpha$ uniformly from $[0,1]$. If $\alpha<A$, set $s=t$. Otherwise, keep $s$ unchanged. 
\end{enumerate}
The justification for its correctness is the same as Claim \ref{cl:tt}.

\subsection{Comments}

We have described the method for AIS and TT. The proposed approach works also for simulated tempering (ST) and parallel tempering (PT). Both ST and PT use the same continuation sequence as AIS. In principle, the approach can be combined with even the simple Metropolis-Hasting, with the transition \eqref{eq:Tgroup} used as a proposal. However, this would require $E_R(s)$ to be very close to $E(s)$ because otherwise, this would not be effective.

The traditional tempering methods utilize the high-temperature distribution as the reference (or at least as the limiting one). It is trivially symmetric and easy-to-sample. However, this high-temperature reference can be quite far from the target distribution. The main observation of our approach is that any distributions with exact symmetries are easy to sample, and there are many more of them. For a target distribution $p(s)$ with approximate symmetry, our method identifies a nearby reference distribution $p_R(s)$ with exact symmetry. The distance between $p(s)$ and $p_R(s)$ is much closer to the one between $p(s)$ and the high-temperature distribution. Therefore, the continuation path can be much shorter.

The idea of using a reference distribution without temperature adjustment has been explored before. In \cite{ying2023annealed}, for the Ising model with boundary condition, the reference distribution is set to be the one with the boundary condition turned off. Still, the distance between the target distribution and the reference one is quite significant.

\section{Examples}\label{sec:ex}

\subsection{Example 1}\label{sec:ex1}

Consider the Ising model on a square lattice of size $n\times n$. Let $I$ be the set of nodes and write $i \sim j$ when a bond links $i$ and $j$. The state space consists of all $\{\pm 1\}^I$ configurations. $\beta$ is a fixed inverse temperature. A forcing term is specified on the four sides as follows.
\[
f_i = \begin{cases}
  -1 + \frac{1}{2} N(0,1) & i\;\text{on the left/right sides,}\\
  +1 + \frac{1}{2} N(0,1) & i\;\text{on the top/bottom sides,}\\
  0 & \text{otherwise.}
\end{cases}
\]
where $N(0,1)$ stands the standard Gaussian.

Given a configuration $s=(s_i)_{i\in I}$, the log density $E(s)$ is defined via
\[
\beta^{-1} E(s) = \sum_{i\sim j} s_i s_j + \sum_i f_i s_i, \quad\text{or}\quad  \beta^{-1} E(s) = \frac{1}{2} s^T As + f^T s,
\]
where the matrix $A$ is the connectivity matrix of the lattice, , i.e., $A_{ij}=1$ iff $i\sim j$ and otherwise $0$.

The forcing term tends to pin down the spins at the left/right sides to be $-1$ and the ones at the top/bottom sides to be $+1$. An essential feature of this Ising model is that the distribution exhibits macroscopically different profiles below the critical temperature. The two dominant macroscopic profiles are a $-1$ cluster linking the left/right sides and a $+1$ cluster linking the top/bottom sides, as shown in Figure \ref{fig:ex1}.
\begin{figure}[h!]
  \centering
  \includegraphics[scale=0.6]{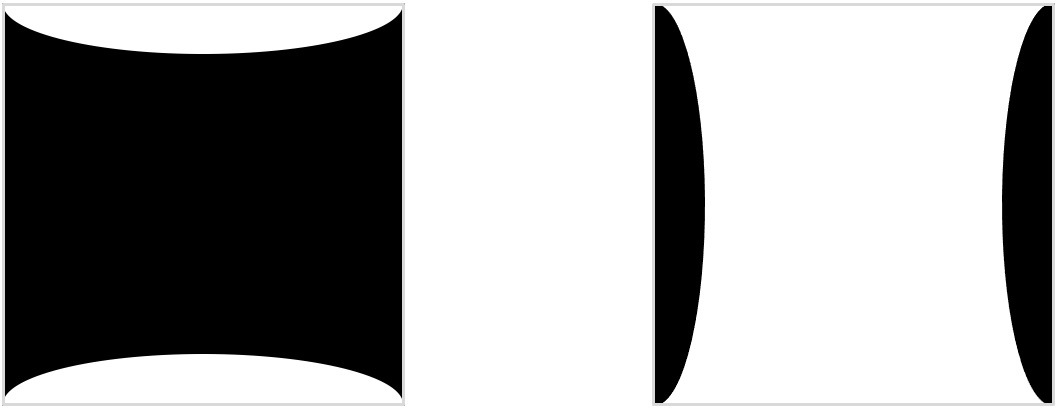}
  \caption{Two dominant macroscopic profiles of the Ising model in Example 1. White stands for $+1$ while black for $-1$. The two dominant macroscopic profiles are a $-1$ cluster linking the left/right sides and a $+1$ cluster linking the top/bottom sides.}
  \label{fig:ex1}
\end{figure}

\subsubsection{Identifying approximate symmetry}
Due to the randomness in $f$, there is no exact symmetry. However, there is an approximate symmetry. Let $G=\{e,g\}$ be the group acting on the nodes $I$, where $g$ acts on a node $i$ by flipping along the $45$-degree line. The key observation is that, for a boundary node $i\in I$,
\begin{equation}
  f_{gi} \approx - f_i. \label{eq:ex1f}
\end{equation}

Based on this group action on the nodes, we lift the action to the configurations: for any configuration $s$,
\[
(es)_i = s_i, \quad (gs)_{gi} = -s_i.
\]
The group element $e$ does not change the configuration as expected. The element $g$ flips the lattice nodes along the 45-degree line and then the spin value. The minus sign is essential since it would otherwise not match the forcing term applied on the boundary nodes. We refer to the action of $g$ as the {\em double-flip}, following \cite{ying2022double}.

The claim is $E(gs) \approx E(s)$. To show this, 
\[
\begin{aligned}
  \beta^{-1} E(gs) & = \sum_{i\sim j} (gs)_i (gs)_j + \sum_i f_i (gs)_i = \sum_{gi \sim gj} (gs)_{gi} (gs)_{gj} + \sum_i f_{gi} (gs)_{gi}\\
  & \approx \sum_{i \sim j} (-s_i) (-s_j) + \sum_i (-f_i)(-s_i) = \beta^{-1} E(s),
\end{aligned}
\]
where the approximation uses \eqref{eq:ex1f}.

\begin{figure}[h!]
  \centering
  \begin{tabular}{cc}
    \includegraphics[scale=0.3]{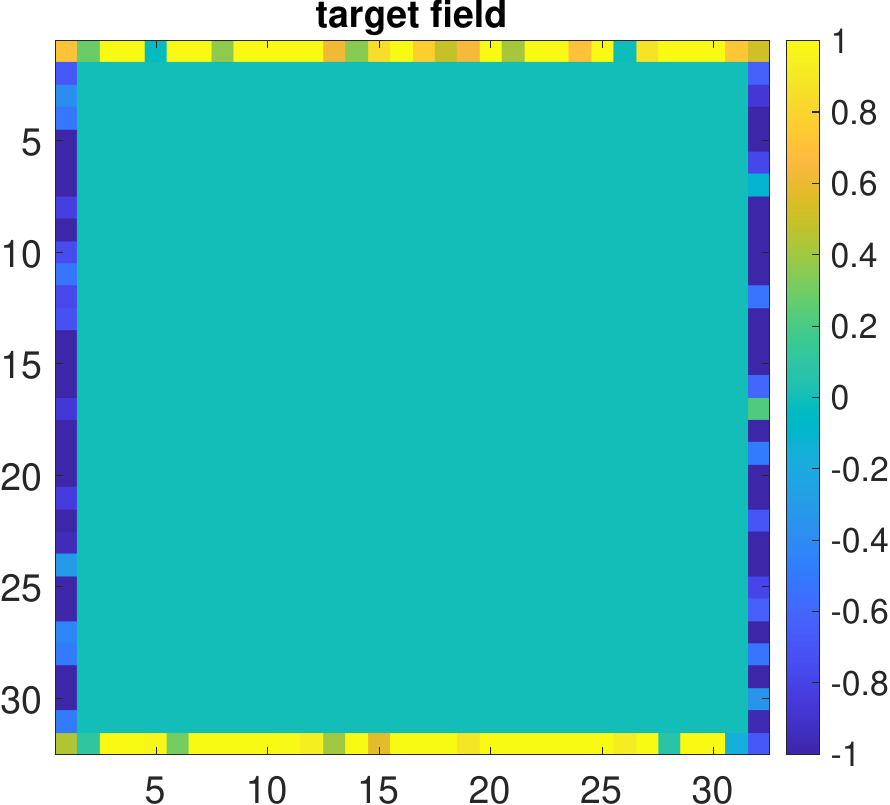} &
    \includegraphics[scale=0.3]{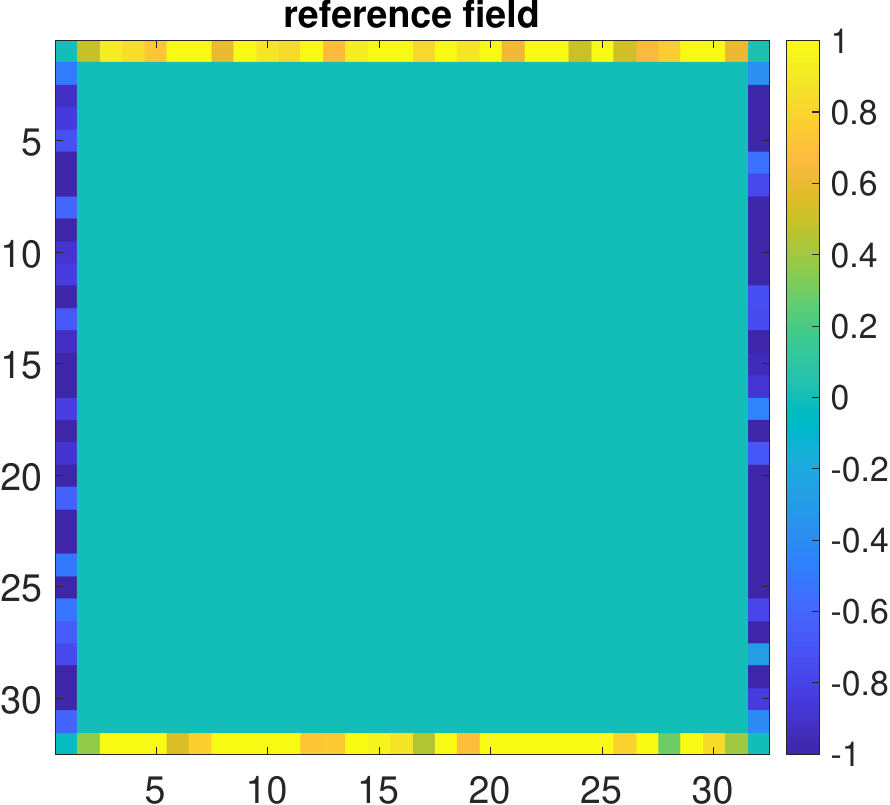} \\
    (a) & (b)
  \end{tabular}
  \caption{Example 1. (a): the forcing term $f$ of $E(s)$. (b): the forcing term $\frac{1}{2} (f-Pf)$ of $E_R(s)$.}
  \label{fig:ex1f}
\end{figure}

Based on that, the reference distribution is defined via 
\[
\beta^{-1} E_R(s) = \frac{1}{2} (E(s) + E(gs)) = \sum_{i\sim j} s_i s_j + \frac{1}{2} \sum_i (f_i - f_{gi}) s_i,
\]
or in terms of the connectivity matrix $A$
\[
\beta^{-1} E_R(s) = \frac{1}{2} s^T A s + \frac{1}{2} (f-Pf)^T s.
\]
Here, $P$ is the permutation matrix associated with the action of $g$ on the nodes. Figure \ref{fig:ex1f} gives the forcing terms: $f$ of $E(s)$ and $\frac{1}{2} (f-Pf)$ of $E_R(s)$ for $n=32$.

In what follows, we present the results at $n=32$ and $\beta=0.8$ for the extensions of AIS and TT mentioned above.

\subsubsection{Extension of AIS}\label{sec:ex1ais}

Applying the method described in Section \ref{sec:pmais} to this example requires specifying the sampling of $E_R(\cdot)$ and the transitions $T_l(\cdot,\cdot)$. Both use the Glauber dynamics done in parallel: we partition the nodes into a collection of independent subsets. These subsets are visited in a random order. The updates for all nodes in a single independent subset are performed in parallel. In what follows, we will refer to this as a {\em Glauber sweep} over the nodes.

First, in order to obtain one sample from the symmetric $E_R(\cdot)$, we choose an initial configuration uniform-randomly and then perform a few Glauber sweeps of $E_R(\cdot)$. The uniform-random initial configuration ensures that the two modes of $E_R(\cdot)$ are visited equally, while the Glauber dynamics mix rapidly in each mode \cite{gheissari2018effect}. Second, the transition $T_l(\cdot,\cdot)$ at each level $l$ is also implemented with a single Glauber sweep of $E_l(\cdot)$.

The linear function $c(t)=t$ is used to interpolate between $E_R(\cdot)$ and $E(\cdot)$, with $L=64$. A total number of $K=10000$ independent weighted samples are collected, with each denoted as $(s^{(k)},w^{(k)})$ for $1\le k \le K$.

To monitor the variance of the algorithm, we record the weight history within each level $l$:
\[
w^{(k)}_l := \frac{p_1(s^{(k)}_{1/2})}{p_0(s^{(k)}_{1/2})} \cdot \cdots \cdot \frac{p_l(s^{(k)}_{l-1/2})}{p_{l-1}(s^{(k)}_{l-1/2})}
\]
for $l=1,\ldots,L$. These weights are then normalized at each level $l$
\begin{equation}
{w}^{(k)}_l \Leftarrow {w^{(k)}_l}/\left(K^{-1} \sum_{g=1}^K w^{(g)}_l\right). \label{eq:wn}
\end{equation}
The sampling efficiency, a quantity between $0$ and $1$, is defined as $(1+\text{Var}(\{{w}^{(k)}_L\}_{k=1}^K))^{-1}$, following \cite{neal2001annealed}. 

Figure \ref{fig:ex1w} plots the logarithms of the normalized weights of each sample as a function of the level $l$. The sampling efficiency is $0.65$, which translates to $98$ Glauber sweeps per independent sample. The probabilities at the $+1$ and $-1$-dominant configurations are estimated to be $67\%$ vs $33\%$, respectively.

\begin{figure}[h!]
  \centering
  \includegraphics[scale=0.3]{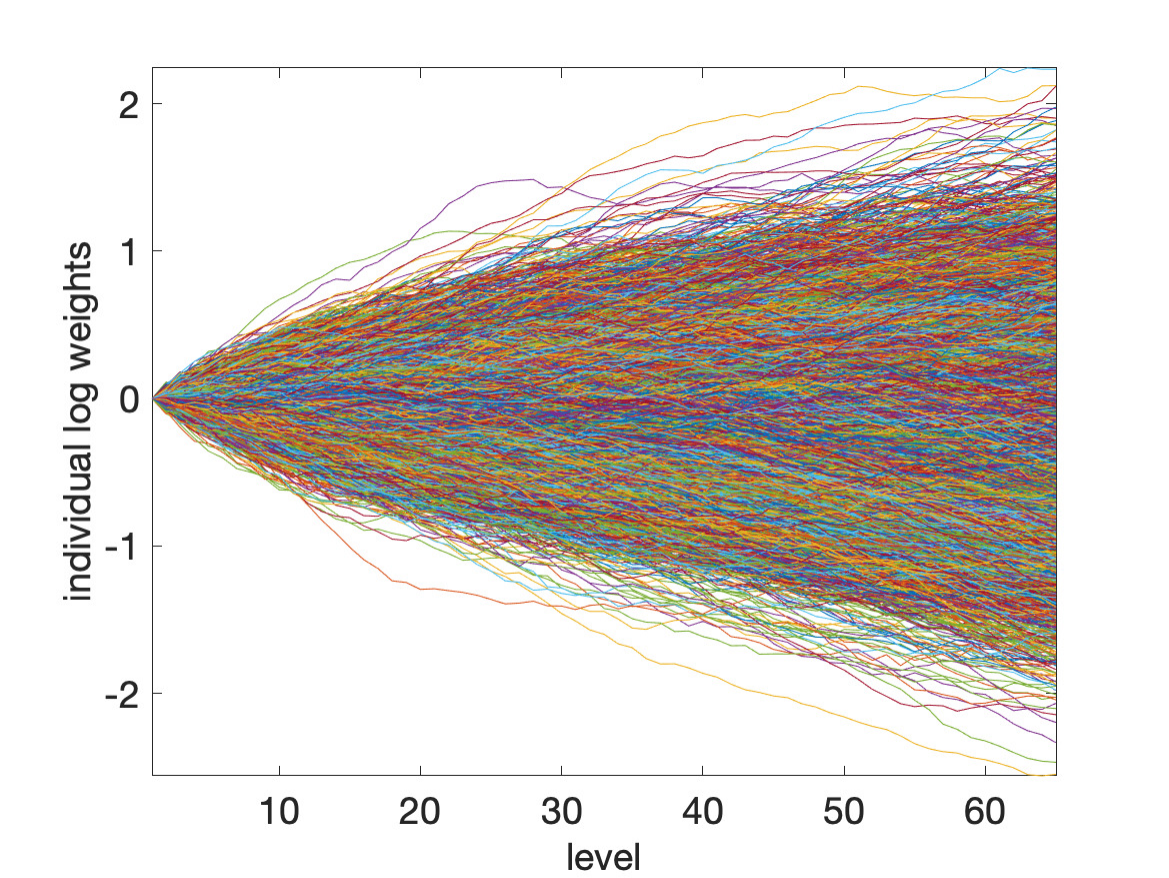}
  \caption{The logarithms of the normalized weights \eqref{eq:wn} of each sample over all levels.}
  \label{fig:ex1w}
\end{figure}

\subsubsection{Extension of TT}\label{sec:ex1tt}

The method described in Section \ref{sec:pmtt} is embedded in an MCMC run. Among the moves, $99\%$ are the (local) parallel Glauber sweeps of $E(s)$, while the other $1\%$ are TT moves described in Section \ref{sec:pmtt}. The percentage is chosen to more or less balance the work between the two parts.

For the TT move, we adopt again the linear interpolation from $E(s)$ to $E_R(s)$ and back, with $L=64$. The transition $T_l(\cdot,\cdot)$ at each level $l$ is also implemented with a single Glauber sweep of $E_l(s)$. A total of $10000$ MCMC moves are performed, and about $100$ TT moves are attempted among them.

In order to monitor the performance, we record the average spin value $ (\sum_{i\in I} s_i)/|I|$. Figure \ref{fig:ex1ave} plots this average as a function of time. There are $70$ successful transitions between the two modes, and the probabilities at $+1$ and $-1$-dominant configurations are estimated to be $65\%$ vs $35\%$, respectively.

\begin{figure}[h!]
  \centering
  \includegraphics[scale=0.3]{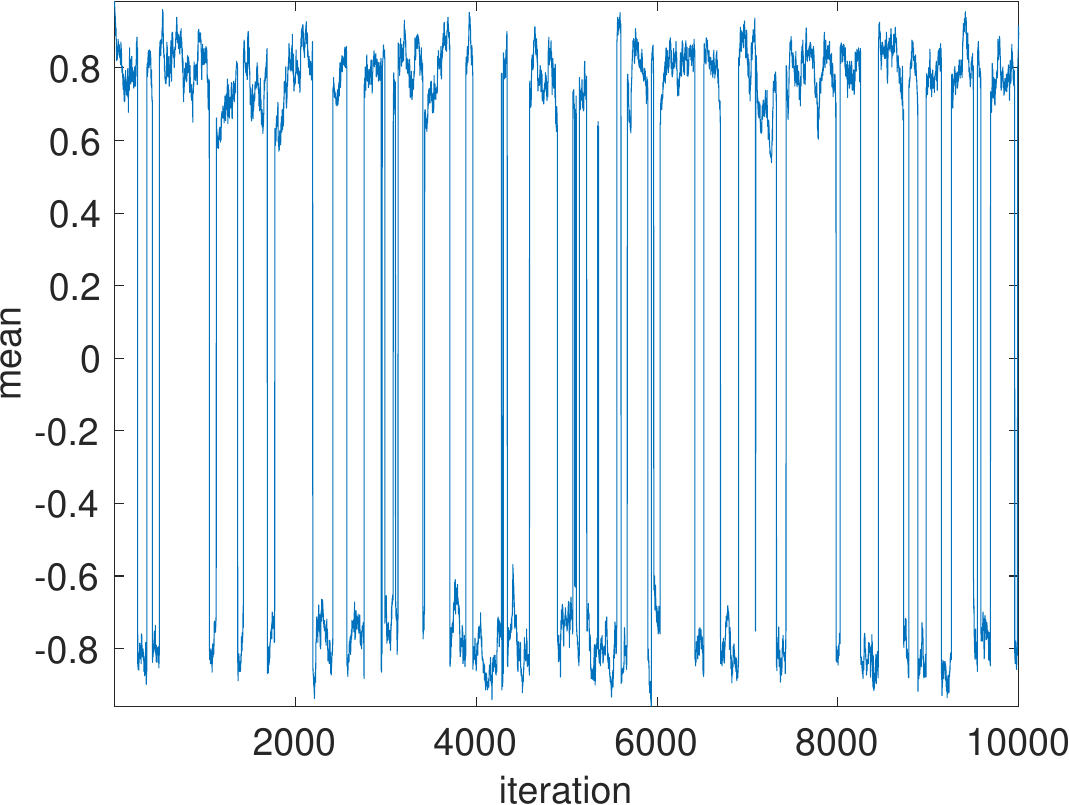}
  \caption{The average spin value along the MCMC simulation. Each jump stands for a successful TT move between the $+1$ and $-1$ dominant configurations.}
  \label{fig:ex1ave}
\end{figure}

\subsection{Example 2}\label{sec:ex2}

Consider the Ising model on a rectangular lattice of size $n_1 \times n_2$ with $n_1\not=n_2$. A forcing term is again specified on the four sides:
\[
f_i = \begin{cases}
  -1 +c, & i\; \text{on the left/right sides,}\\
  +1 +c, & i\; \text{on the top/bottom sides,}\\
  0, & \text{otherwise,}
\end{cases}
\]
where the constant $c$ is chosen so that $f_i$ has zero mean. Ensuring $\sum_i f_i=0$ helps to make both $+1$-dominant and $-1$-dominant profiles have non-negligible probabilities.

Given a configuration $s=(s_i)_{i\in I}$, the log probability $E(s)$ is defined via
\[
\beta^{-1} E(s) = \sum_{i\sim j} s_i s_j + \sum_i f_i s_i, \quad\text{or}\quad  \beta^{-1} E(s) = \frac{1}{2} s^T As + f^T s.
\]
Similar to Example 1, the distribution with the specified forcing term exhibits macroscopically different profiles below the critical temperature.

\subsubsection{Identifying approximate symmetry}

Since $n_1 \not=n_2$, there is no exact symmetry. However, there is an approximate symmetry. The group is still $G=\{e,g\}$. We start by identifying the action of $g$ on the nodes. we embed the nodes $I$ into the $[-1,1]^2$ box, with stepsize $\frac{2}{n_1-1}$ in the horizontal direction and $\frac{2}{n_2-1}$ in the vertical one. For each $i$, denote $\vec{x}_i$ as the associated geometric location and correspondingly define $\vec{y}_i$ as
\[
\vec{y}_i = \begin{bmatrix} 0 & 1\\ 1 & 0 \end{bmatrix} \vec{x}_i.
\]
The goal is to find a map $g: I \rightarrow I$ such that two conditions: $g^2 = e$ and
\begin{equation}
  \vec{y}_i \approx \vec{x}_{gi}. \label{eq:xy}
\end{equation}
The motivations are twofold. First, such $g$ would ensure $f_{gi} \approx -f_i$. Second, if $i\sim j$, then $\vec{x}_i \approx \vec{x}_j$ and $\vec{y}_i \approx \vec{y}_j$. \eqref{eq:xy} implies that $\vec{x}_{gi} \approx \vec{x}_{gj}$, which further correlates with $gi \sim gj$, i.e., the nodes $gi$ and $gj$ are linked in the Ising model. We adopt the following heuristic procedure proposed in \cite{ying2022double} to achieve these two conditions.
\begin{enumerate}
\item Order the interior vertices $\{\vec{x}_j\}_{j\in I}$ based on their distances to the origin in the
  decreasing order. The distance is typically chosen to be either the $\ell_\infty$ norm or the
  $\ell_2$ norm.
\item Mark all vertices $j\in I$ as unpaired.
\item Scan the interior vertices in this ordered list. For each $\vec{x}_j$, if $j$ is already paired, then
  skip. If not, find the unpaired $i$ such that $\vec{y}_i$ is closet to $\vec{x}_j$, pair $i$ and $j$
  \[
  gi:=j, \quad gj:=i,
  \]
  and mark both $i$ and $j$ as paired. 
\end{enumerate}
The heuristic is that, by following the order of decreasing distance to the origin, the remaining unpaired vertices are forced to cluster near the center of the domain, thus reducing the distance between them.

We then lift the action of $g$ from the nodes to the configuration $s$. Given $s$, we define $gs$ by
\[
(es)_i = s_i, \quad (gs)_{(gi)} = -s_i.
\]
The minus sign here is again essential for $g$ to be an approximate symmetry.

The claim is $E(gs) \approx E(s)$. To see this,
\[
\begin{aligned}
  \beta^{-1} E(gs) & = \sum_{i\sim j} (gs)_i (gs)_j + \sum_i f_i (gs)_i \approx \sum_{gi \sim gj} (gs)_{gi} (gs)_{gj} + \sum_{gi} f_{gi} (gs)_{gi}\\
  & \approx \sum_{i\sim j} (-s_i) (-s_j) + \sum_i (-f_i) (-s_i) = \beta^{-1} E(s),
\end{aligned}
\]
where the first approximation uses the fact that $i\sim j$ is correlated with $gi\sim gj$ and the second approximation uses $f_{gi} \approx -f_i$.

Based on that, the reference distribution $E_R(s)$ is defined as 
\[
\beta^{-1} E_R(s) = \frac{1}{2} \left(\sum_{ij} s_i s_j + \sum_i f_i s_i \right) + \frac{1}{2} \left( \sum_{ij} (gs)_i (gs)_j + \sum_i f_i (gs)_i \right).
\]
or in terms of the connectivity matrix $A$
\[
\beta^{-1} E_R(s) = \frac{1}{2} s^T (A + P^T A P) s + \frac{1}{2}(f-Pf)^T s.
\]
Here, $P$ is again the permutation matrix associated with the action of $g$ on the nodes. Figure \ref{fig:ex2Af} plots the $g$ action between the nodes, the connectivity of $E_R(s)$, and the forcing terms of $E(s)$ and $E_R(s)$, for $n_1=30$, $n_2=32$.  Notice that in this example, $P^T A P$ differs from $A$; therefore, the connectivity of $E_R(s)$ has more bonds.

In what follows, we present the results for $n_1 = 30$, $n_2=32$, and $\beta = 0.8$ for the extensions of AIS and TT.

\begin{figure}
  \centering
  \begin{tabular}{cc}
    \includegraphics[scale=0.3]{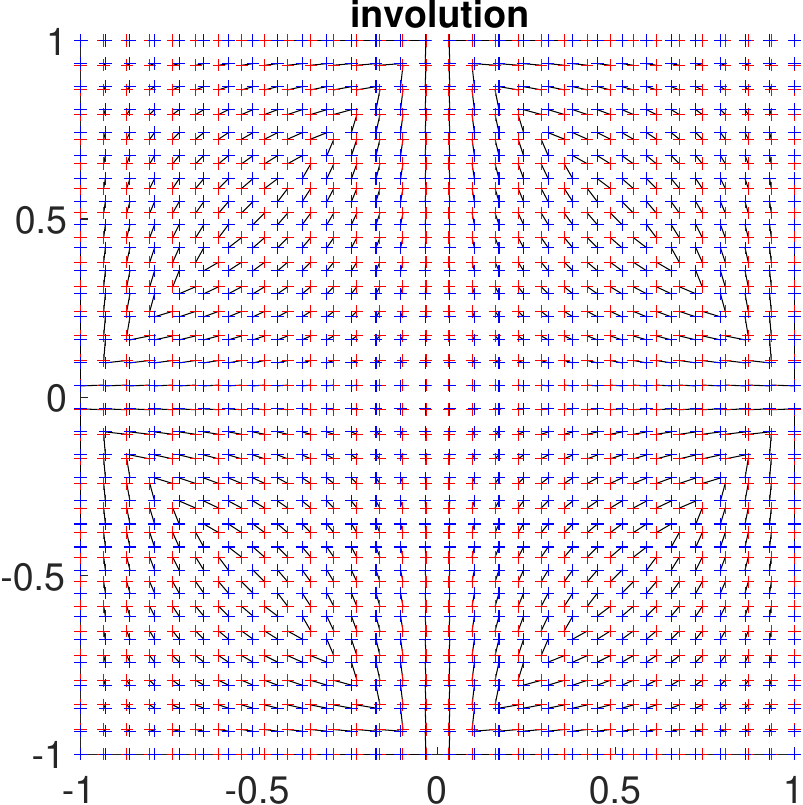} &    \includegraphics[scale=0.3]{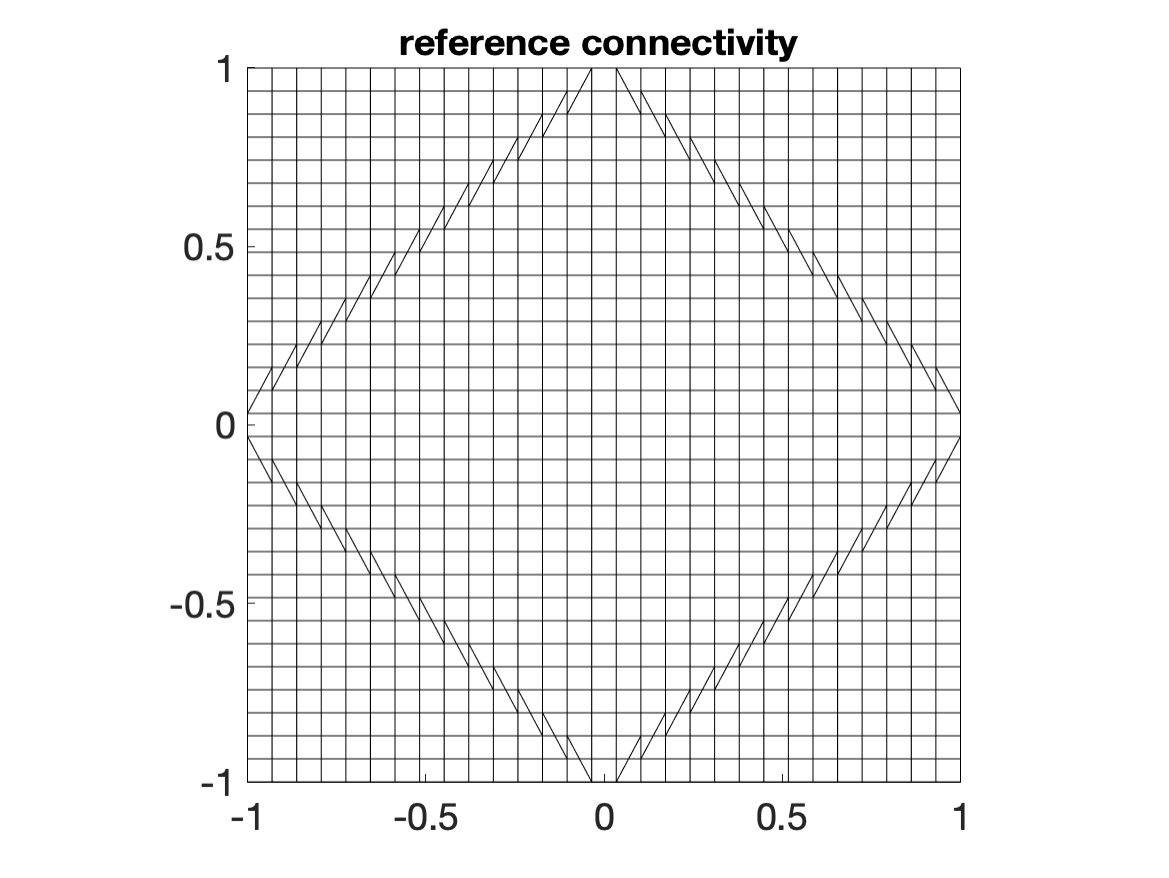} \\
    (a) & (b) \\
    \includegraphics[scale=0.3]{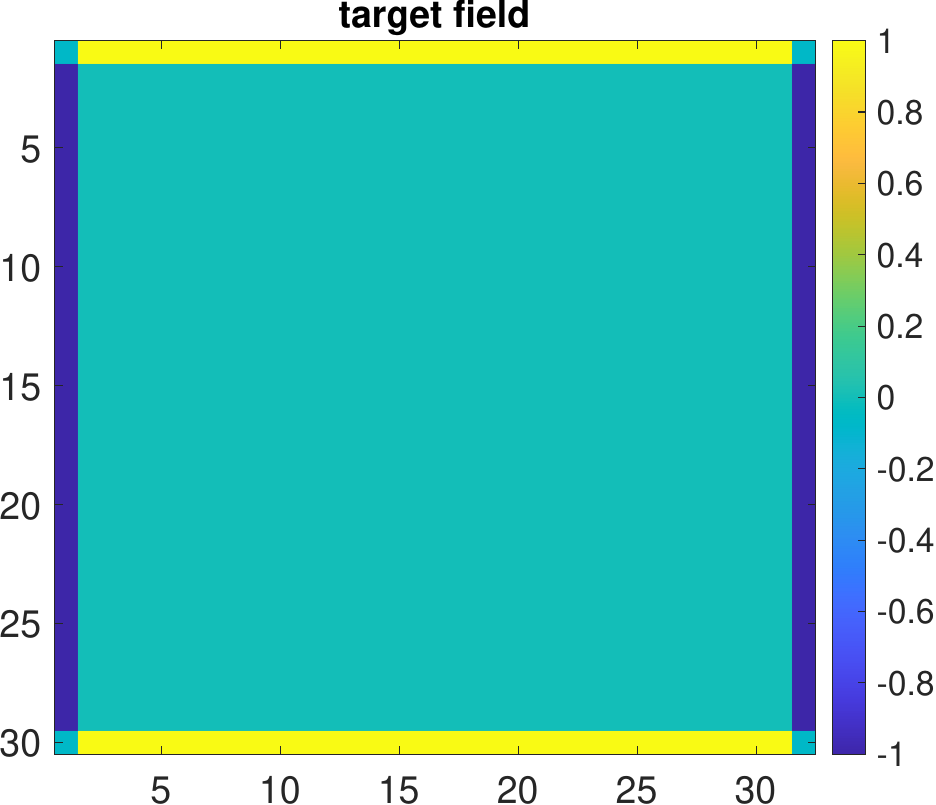} &    \includegraphics[scale=0.3]{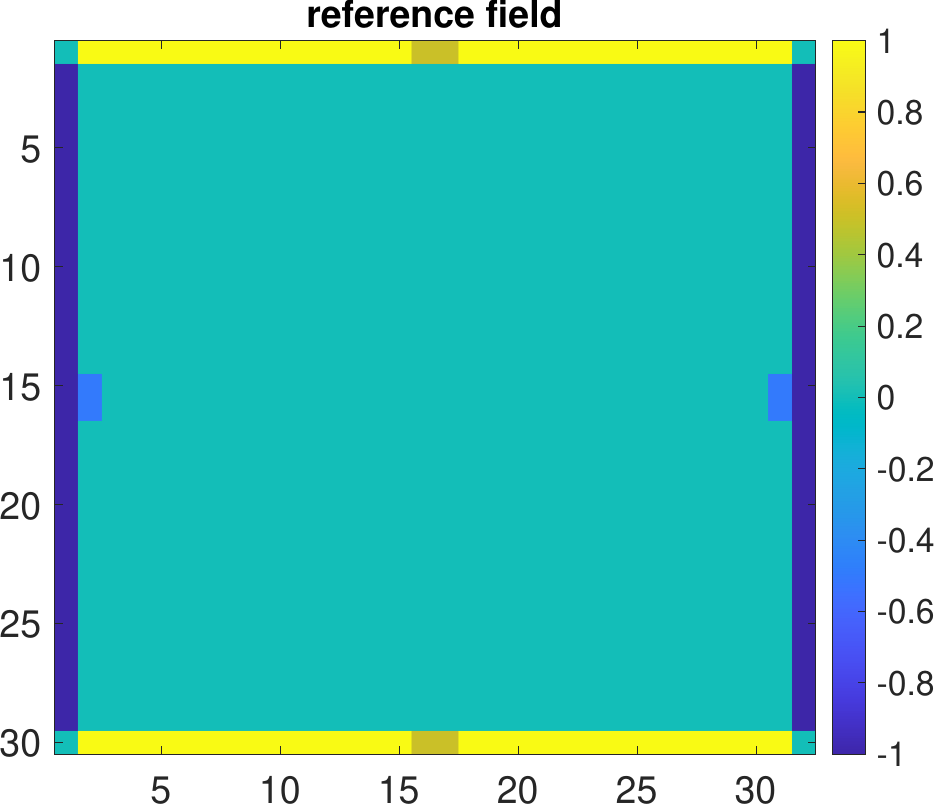} \\
    (c) & (d) 
  \end{tabular}
  \caption{Example 2.  (a): the map $g:I \rightarrow I$. The points $\{\vec{x}_i\}$ and $\{\vec{y}_i\}$ are marked in blue and red, respectively.  A black segment is linked between $\vec{y}_i$ and $\vec{x}_j$ if and only if $gj=i$ and $gi=j$.  (b): the connectivity pattern of $\frac{1}{2}(A + P^T AP)$. Notice that it contains a few more edges than the connectivity pattern of $A$ due to the averaging over the group orbit.  (c): the forcing term $f$ of $E(s)$. (d): the forcing term $\frac{1}{2}(f-Pf)$ of $E_R(s)$.}
  \label{fig:ex2Af}
\end{figure}

\subsubsection{Extension of AIS}\label{sec:ex2ais}

We adopt a linear interpolation between $E_R(\cdot)$ and $E(\cdot)$, with $L=64$. The choice of the initial configuration and the transition $T_l(\cdot,\cdot)$ are the same as Section \ref{sec:ex1ais}. A total of $10000$ samples are collected in this test.

Figure \ref{fig:ex2w} gives the logarithm of the normalized weights for each sample across all levels. The sampling efficiency is $0.49$, which translates to $129$ Glauber sweeps per independent sample. The probabilities at $+1$ and $-1$-dominant configurations are estimated to be $98\%$ vs $2\%$. This difference between the probability masses of the two modes originates from the slight discrepancy between $n_1$ and $n_2$ but is enhanced by the relatively large $\beta$ value. About half of the samples are devoted to the $98\%$ $+1$-dominant configurations and the other half to the $2\%$ $-1$-dominant ones. Given prior knowledge about this imbalance, this inefficiency can be improved by lowering the probability but increasing the weights for the initial samples obtained from $E_R(s)$.

\begin{figure}[h!]
  \centering
  \includegraphics[scale=0.3]{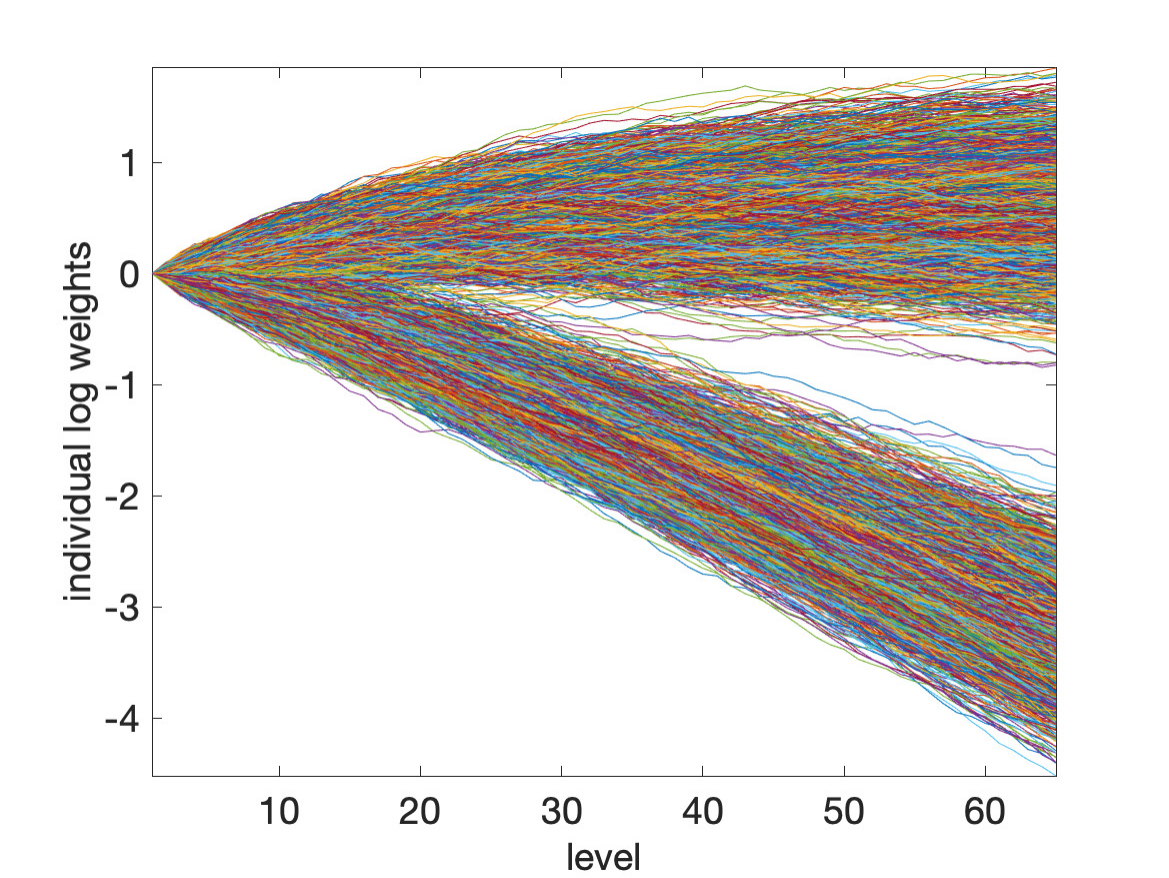}
  \caption{The logarithms of the normalized weights \eqref{eq:wn} of each sample over all levels.}
  \label{fig:ex2w}
\end{figure}

\subsubsection{Extension of TT}\label{sec:ex2tt}

The method described in Section \ref{sec:pmtt} is embedded in an MCMC run. Among the moves, $99\%$ are the parallel Glauber sweeps of $E(s)$, while the remaining $1\%$ are TT moves.

We use the linear interpolation for the TT move from $E(\cdot)$ to $E_R(\cdot)$ and back, with $L=128$. The transition $T_l(\cdot,\cdot)$ at each level $l$ is implemented also with a Glauber sweep of $E_l(\cdot)$.  A total of $10000$ moves are carried out, with about $100$ TT moves being attempted.

Figure \ref{fig:ex2ave} plots the average spin value as a function of time. There are $36$ successful transitions between the modes, and the probabilities at $+1$ and $-1$-dominant configurations are estimated to be $98\%$ vs $2\%$.

\begin{figure}[h!]
  \centering
  \includegraphics[scale=0.3]{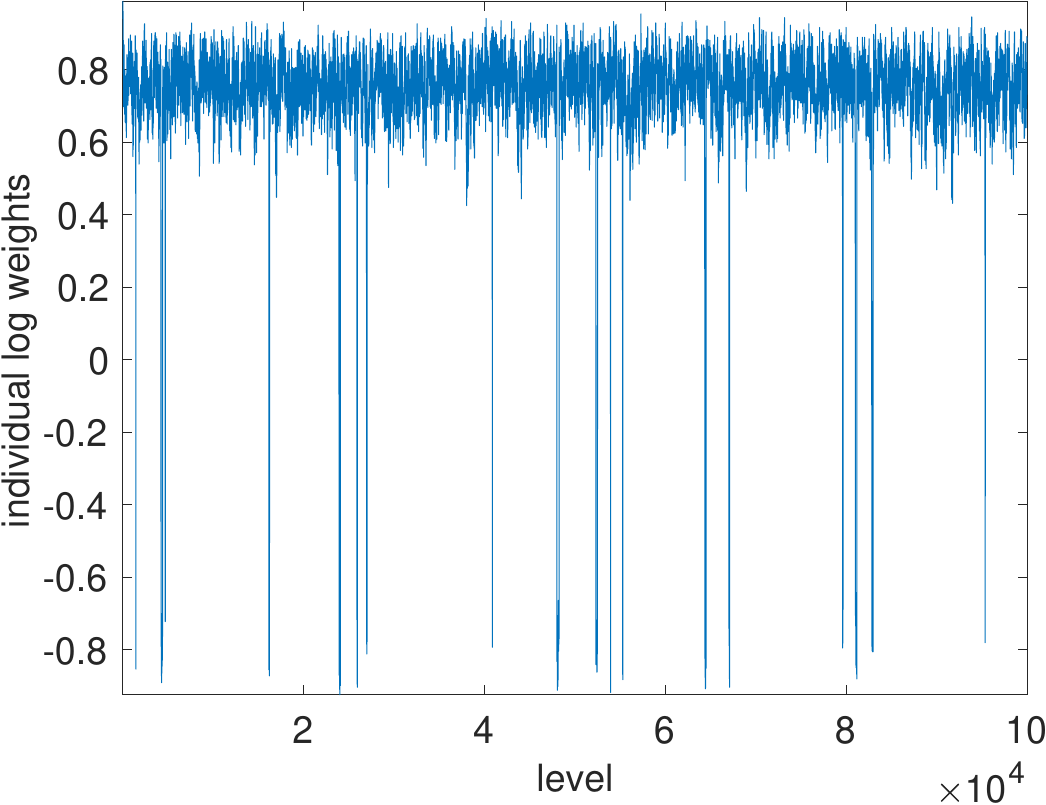}
  \caption{The average spin value along the MCMC simulation. Each jump stands for a successful transition between the $+1$ and $-1$ dominant configurations.}
  \label{fig:ex2ave}
\end{figure}

\section{Discussion}\label{sec:disc}

This paper proposed a new approach for sampling multimodal distributions. The main requirement is the identification of an approximate symmetry of the target distribution. Once the approximate symmetry is identified, we construct a nearby symmetric reference distribution and combine it with two continuation methods, AIS and TT. This framework can also be combined with simulated tempering (ST) and parallel tempering (PT). This is a clear direction for future work.

This paper considered only discrete state space. The method should also work for distributions equipped with an approximate symmetry over a continuous state space. The proposed method only requires a little change when the group is discrete. When the group is a Lie group, one might need to consider averaging the group action over the Haar measure for the definition of $E_R(s)$, which is a direction for future work.

\bibliographystyle{abbrv}

\bibliography{ref}

\end{document}